\documentclass[10pt]{amsart}
\usepackage{amssymb}
\usepackage{bm}
\usepackage{graphicx}
\usepackage[centertags]{amsmath}
\usepackage{amsfonts}
\usepackage{amsthm}
\usepackage{comment}

\usepackage{pgf}
\usepackage{pgfkeys}
\usepackage{ifpdf}
\usepackage{color}

\newtheorem{thm}{Theorem}
\newtheorem{claim}[thm]{Claim}
\newtheorem{sublem}[thm]{Sublemma}
\newtheorem{cor}[thm]{Corollary}
\newtheorem{lem}[thm]{Lemma}

\newtheorem{defn}[thm]{Definition}
\theoremstyle{definition}

\newtheorem{fact}[thm]{Fact}

\newcommand{\nn}{\mathbb{N}}

\newcommand{\vs}{\vec{s}}
\newcommand{\vt}{\vec{t}}
\newcommand{\vw}{\vec{w}}
\newcommand{\V}{V^\infty (A)}

\newtheorem*{thm*}{Theorem}

\begin{document}
\title[Comb. proofs of Infinite versions of the  Hales--Jewett theorem]
{Combinatorial proofs of Infinite versions of the  Hales--Jewett
theorem}
\author{Nikolaos Karagiannis}
\address{National Technical University of Athens, Faculty of Applied Sciences,
Department of Mathematics, Zografou Campus, 157 80, Athens, Greece}
\email{nkaragiannis@math.ntua.gr}


\begin{abstract} We  provide new and purely combinatorial proofs of
two infinite  extensions of the
Hales--Jewett theorem. The first one
is due to T. Carlson and S. Simpson  and the second one is due T.
Carlson. Both concern  infinite increasing sequences of finite
alphabets.
\end{abstract}

\maketitle


\section{introduction}
The aim of this note is to provide  purely combinatorial  proofs
of two known
 infinitary
extensions of the Hales--Jewett  theorem. To state the theorems we
need first to recall the relevant terminology. By $\nn=
\{0,1,...\}$ we denote the set of all non negative integers. Let
$A$ be an  \textit{alphabet} (i.e. any non empty  set). The
elements of $A$ will be called letters. By $W(A)$ we denote the
set of \textit{constant words} \textit{over} $A$, that is all
finite sequences with elements in $A$ including the empty
sequence. For $N \in \nn$, by $A^N$, we denote all finite sequences
from $A$, consisting of $N$ letters. We also fix a variable  $x \notin A$,
then a \textit{variable word} \textit{over} $A$ is an element in $W(A
\cup \{x\}) \setminus W(A)$. The variable words will be denoted by
$s(x), t(x), w(x)$, etc. A \textit{left variable} word is a
variable word  such that $x$ is its leftmost letter. Given a
variable word $s(x)$ and $a\in A$, by $s(a)$ we denote the constant
word in $W(A)$ resulting from the substitution of  the variable
$x$ with    the letter $a$. Given $q\in\nn$ with $q\geq 1$, then a
$q$-\textit{coloring} of a set $X$ is any map $c:X \rightarrow
\{1,...,q\}$. A subset $Y$ of $X$ will be called
\textit{monochromatic}, if there exists $1\leq i \leq q $ such
that $c(y)=i$, for all $y \in Y$. Finally, for every finite set
$X$, by $|X|$ we denote its cardinality.

 The following is the Hales--Jewett
theorem \cite{H-J}.

\begin{thm}\label{thm1} For every positive integers $p,q$ and every finite
alphabet $A$
  with $|A|=p$, there exists\textbf{} a positive integer
  $N_0=HJ(p,q)$ with the following property. For
  every $N \geq N_0$ and for every $q$-coloring of
  $A ^N$ there exists a variable word $w(x) \in (A\cup \{x\})^N
  \setminus A^N$ such that the set
  $\{w(a):a \in A\}$ is monochromatic.
\end{thm}

The Hales--Jewett theorem gave birth to a whole new branch of
research which concerns infinite extensions of it in the context
of both  finite and infinite alphabets (see \cite{C}, \cite{C-S},
\cite{F-K}, \cite{HMC}, \cite{MC}, \cite{PV}). For an
exposition of those results the reader can also refer to
 \cite{MCb}, \cite{Tod}.

The  first  theorem that we will prove is  due to T. Carlson and
S. Simpson \cite{C-S}.

\begin{thm} \label{thm2}
 Let $(A_n)_{n=0}^ \infty $ be an increasing  sequence of finite
 alphabets and let $A=\cup_{n\in \nn} A_n$.
 Then for every finite coloring of $W(A)$ there exists a sequence
  $(w_n(x))_{n=0}^ \infty$ of variable words over $A$
  such that for every $n\geq1$, $w_n(x)$ is a left variable word
  and for every $n\geq 0$ the words of the form $w_0(a_0) w_1(a_1) ... w_n(a_n)$ with
 $a_0 \in A_{0}, a_1\in A_{1}, ... , a_n \in A_{n}$ are of the same color.
\end{thm}
The case where $A_n=A$ for all  $n \in \nn$ of the above theorem, was the
key lemma for the main result of \cite{C-S} and is commonly referred
as the Carlson--Simpson theorem. As noted in
\cite[p. 274]{C-S}, the above reformulation in terms of an infinite
increasing  sequence of finite alphabets, is due to Miller and
Prikry and it is closely related  to an infinite extension of the
Halpern--L\"{a}uchli theorem due to R. Laver \cite{L}.

The second  theorem that we will prove is  due to T. Carlson
\cite{C}.
\begin{thm}\label{thm3}
  Let $(A_n)_{n=0}^ \infty$  be an increasing  sequence of finite
  alphabets and let $A=\cup_{n \in \nn} A_n$.
 Then for every finite coloring of $W(A)$ there exists a sequence
  $(w_n(x))_{n=0}^ \infty$ of variable words over $A$
  such that for every $n\in\nn$ and every  $m_0<m_1<...<m_n$,
  the  words of the form $w_{m_0}(a_0) w_{m_1}(a_1) ... w_{m_n}(a_n)$ with
 $a_0 \in A_{m_0}, a_1\in A_{m_1}, ... , a_n \in A_{m_n}$ are of the same color.
\end{thm}

Actually, Theorem \ref{thm3} is a consequence of a more general
result of T. Carlson (see \cite[Theorem 15]{C}). As shown in
\cite[\S 3]{HMC}, a left variable version of Theorem \ref{thm3} is
not true. However, such a version holds true for the case of a
finite alphabet \cite[Theorem 2.3]{MC}.

 The common approach of
the proof of the above results is based on topological as well as
algebraic notions of the Stone--{\v C}ech compactification of the
related structures (see e.g. \cite{Hb}). On the other hand, the
proofs of Theorems \ref{thm2} and \ref{thm3} that we are going to
present are based on the classical Hales--Jewett theorem and avoid
the use of ultrafilters. Our approach has its origins in the proof
of Hindman's theorem \cite{H} due to J. E. Baumgartner \cite{B}.
Actually,  a proof of a weaker version of Theorem \ref{thm2} given
in \cite[\S 2.3]{MCb} was the motivation for this note. The main
difficulty that we encountered was the manipulation of the
infinite sequence $(A_n)_{n=0}^ \infty$ of alphabets. In
particular, we remark that if one wishes to prove the
aforementioned results for a finite alphabet then the proofs are
considerably simpler.

\section{Proof of theorem \ref{thm2}}

\subsection{Preliminaries.} \label{Initializing the
proof}  In this subsection we introduce some notation and
terminology that we will use for the proof of Theorem \ref{thm2}.
 We fix  an increasing  sequence
 $$A_0\subseteq A_1 \subseteq ... \subseteq A_n \subseteq ...$$
 of finite alphabets and we set $$A=\cup_{n\in \nn} A_n.$$
 Let
 $V(A)$ be the set of all \textit{variable words} (\textit{over}
 $A$).
 By $V^ {<\infty}(A) $ (resp. $V^ \infty(A)
$) we denote the set of all finite (resp. infinite) sequences of
variable words. Also let $V^{\leq\infty}(A)=V^ {<\infty}(A) \cup
V^ \infty(A)$. Generally, the elements of $V^{\leq\infty}(A)$ will
be denoted by $\vs,\vt, \vw$, etc.

 \subsubsection{Reduced constant and variable span of a sequence of variable words.}
 Let  $m\in \nn$,
$(s_n(x))_{n=0}^m\in V^{<\infty}(A)$ and $(k_n)_{n=0}^m$ be a
strictly increasing finite sequence of
non negative integers.
 The \textit{reduced constant span of}
$(s_n(x))_{n=0}^m$ \textit{with respect to} $(A_{k_n})_{n=0}^m$ is
defined to be the set \[[(s_n(x))_{n=0}^m \  \| \
(A_{k_n})_{n=0}^m]_c= \{s_0(a_0)s_1(a_1)... s_m(a_m):   a_i\in
A_{k_i} \  \text{for all} \  0\leq i \leq m \}.\]

We also define
the \textit{reduced variable span of} $(s_n(x))_{n=0}^m$
\textit{with respect to} $(A_{k_n})_{n=0}^m$ to be the set
\[[(s_n(x))_{n=0}^m \  \| \ (A_{k_n})_{n=0}^m]_v =
V(A) \cap [(s_n(x))_{n=0}^m \  \| \ (A_{k_n} \cup
\{x\})_{n=0}^m]_c\] that is the reduced variable span of
$(s_n(x))_{n=0}^m$ with respect to $(A_{k_n})_{n=0}^m$ consists of
all variable words of the form $s_0(b_0) ... s_m(b_m)$ such that
for each $0\leq n \leq m$, $b_n \in A_{k_n} \cup \{x\}$ and for at
least one $n$ we have $b_n=x$.

Finally, for a non empty finite subset $B$ of $A$ we set
\[[(s_n(x))_{n=0}^m\;\|\;B]_c=\{s_0(b_0)s_1(b_1)...s_m(b_m):
b_i\in B \text{ for all }
\  0\leq i \leq m \}\] and
\[[(s_n(x))_{n=0}^m\;\|\;B]_v=V(A)\cap
[(s_n(x))_{n=0}^m\;\|\;B\cup \{x\}]_c
.\]

The above notation is naturally extended to infinite sequences of
variable words as follows. Let $(s_n(x))_{n=0}^ \infty  \in \V$
and $(k_n)_{n=0}^ \infty$
be a strictly increasing sequence of non negative integers. Then
the reduced constant span of $(s_n(x))_{n=0}^ \infty$ with respect to
$(A_{k_n})_{n=0}^ \infty$, denoted by $[(s_n(x))_{n=0}^ \infty \  \| \ (A_{k_n})_{n=0}^ \infty]_c$
is the set
\[\begin{split}
& \{s_0(a_0)s_1(a_1) ...s_m(a_m): m\in \nn, a_i\in A_{k_i}  \text{ for all }
0\leq i\leq m \} \\
 &= \cup _{m\in \nn} [(s_n(x))_{n=0}^m \  \| \ (A_{k_n})_{n=0}^m]_c
\end{split}\]
and the reduced variable span of $(s_n(x))_{n=0}^ \infty$ with respect to
$(A_{k_n})_{n=0}^ \infty$ is the set
\[\begin{split}
[(s_n(x))_{n=0}^ \infty \  \| \ (A_{k_n})_{n=0}^ \infty]_v &
=V(A) \cap [(s_n(x))_{n=0}^ \infty \ \| \ (A_{k_n} \cup \{x\})_{n=0}^ \infty]_c \\
&= \cup _{m\in \nn} [(s_n(x))_{n=0}^m \  \| \ (A_{k_n})_{n=0}^m]_v.
\end{split}\]

In the following  we will also write $[\vs \  \| \ (A_{k_n})_{n=0}^ \infty]_c$
(resp. $[\vs \  \| \ (A_{k_n})_{n=0}^ \infty]_v$) to denote the the reduced
constant (resp. variable) span of $\vs=(s_n(x))_{n=0}^ \infty$ with respect to
$(A_{k_n})_{n=0}^ \infty$.

\subsubsection{Reduced $k$-block subsequences  of a sequence of variable words.}
We need to specify a notion of
a ``block subsequence" of a sequence of variable words. To this
end we give the following definition.

\begin{defn}\label{block} Let $k
\in \nn$.
\begin{enumerate}
\item[(i)]Let $l\in \nn$, $\vt=(t_n(x))_{n=0}^l\in V^{<\infty}(A)$
and $\vs=(s_n(x))_{n=0}^ \infty\in \V$. We say that $\vt$ is a (finite)
reduced $k$-block subsequence of $\vs$ if there exist $0=m_0<...<m_{l+1}$ such
that
$$t_i(x)\in [(s_n(x))_{n=m_{i}}^{m_{i+1}-1} \ \| \
(A_{k+n})_{n=m_{i}}^{m_{i+1}-1}]_v,$$ for all  $0 \leq i \leq l$.
 \item[(ii)] Let $\vt=(t_n(x))_{n=0}^\infty$, $\vs=(s_n(x))_{n=0}^\infty \in \V$.
 We say
that $\vt$ is a (infinite) reduced $k$-block subsequence of $\vs$ if for every
$l\in \nn$ the finite sequence $(t_n(x))_{n=0}^l$ is a (finite) reduced
$k$-block subsequence of $\vs$.
\end{enumerate}

\end{defn}
In the following we will write  $\vt \unlhd_k \vs$ whenever $\vt
\in V^{\leq \infty}(A)$, $\vs\in \V$ and $\vt$ is a reduced
$k$-block subsequence of $\vs$.

Taking into account that  the sequence of alphabets $(A_n)_{n=0}^ \infty$ is
increasing, the next facts follow easily from the above
definitions.
\begin{fact}\label{Cf7}
 Let $k \in \nn$ and $\vs,\vt  \in \V$. If  $\vt \unlhd _k
\vs$ then   $$[\vt \; \| \; (A_{k+n})_{n=0}^ \infty]_c \subseteq [\vs\;\|\;
  (A_{k+n})_{n=0}^ \infty]_c \text{ and  } [\vt \; \| \; (A_{k+n})_{n=0}^ \infty]_v
  \subseteq [\vs\;\|\;
(A_{k+n})_{n=0}^ \infty]_v.$$
\end{fact}
\begin{fact}\label{Cf8}
 Let $k_0,k_1 \in \nn$, $\vs,\vt, \vw \in \V$. If
 $k_0\leq k_1$ and  $\vw \unlhd _{k_0} \vt\unlhd_{k_1} \vs$ then  $\vw \unlhd_{k_1} \vs.$
\end{fact}

We close this section with a fusion lemma that  we shall need
later on.
\begin{lem}\label{Cf9}
 Let $(\vt_n)_{n=0}^ \infty$ be a sequence  in  $ \V$ and $(w_n(x))_{n=0}^ \infty$
 be a sequence of variable
 words. Also let  $(k_n)_{n=0}^\infty$ and $(m_n)_{n=0}^\infty$ with $m_0=0$ be two
 sequences  in $
 \nn$. Let $\vt_n=(t_i^{(n)}(x))_{i=0}^\infty$ for all
 $n\in\nn$ and assume that for every $n \geq 1 $ the following are satisfied.
\begin{enumerate}
  \item[(i)] $m_{n} \geq 1$ and $k_{n}=k_{n-1}+m_{n}$.

 \item[(ii)]$w_{n-1}(x) \in [( t_i^{(n-1)}(x))_{i=0}^{m_{n}-1}\; \|\;
  (A_{k_{n-1}+i})_{i=0}^{m_{n}-1} ]_v$.

 \item[(iii)]$\vt _{n}$ is a reduced $k_{n}$-block subsequence of
 $(t_{i}^{(n-1)}(x))_{i=m_{n}}^\infty$.
\end{enumerate}
Then there exists a strictly increasing sequence $(p_n)_{n=0}^
\infty$ in $ \nn$ with $p_0=0$ such that for every $n \in \nn$ the
following are satisfied.
\begin{enumerate}

\item[(C1)] $k_0+p_n \geq k_n$.

\item[(C2)] $w_n(x) \in [(t_i^{(0)}(x))_{i=p_n}^{p_{n+1}-1} \; \|
\; (A_{k_0+i})_{i=p_n}^{p_{n+1}-1}]_v$.

\item[(C3)] $(w_i(x))_{i=n}^\infty$ is a reduced $k_n$-block
subsequence of $ \vt_n$.

\end{enumerate}
\end{lem}

For the proof of Lemma \ref{Cf9}, we shall need the following.

\begin{sublem}\label{extralemma}
  Let $\vt= (t_n(x))_{n=0}^\infty, \vt ^{\;\prime} =(t^\prime _n(x))_{n=0}^\infty \in \V$
  and let $k, k^\prime, p_1, p_1^\prime$ in $\nn$
  satisfying the following properties.
  \begin{enumerate}
    \item[(i)] $k+p_1 \leq k^\prime +p_1 ^\prime$.
    \item[(ii)] $ (t_i(x))_{i=p_1}^\infty \unlhd_{k^\prime + p_1
    ^\prime} (t^\prime _i(x))_{i=p_1^\prime}^\infty$.
  \end{enumerate}
Let $p_2 \in \nn$ with $p_2 > p_1$. Then there exists a $p_2
^{\;\prime} \in \nn$  with $p_2 ^{\;\prime} > p_1 ^\prime$  such
that the following hold.
\begin{enumerate}
  \item[(a)] $ p_2  - p_1\leq p_2 ^\prime - p_1 ^\prime$.
   \item[(b)]$k+p_2 \leq k^\prime +p_2 ^\prime$.
\item[(c)]$ (t_i(x))_{i=p_2}^\infty \unlhd_{k^\prime + p_2
    ^\prime} (t^\prime _i(x))_{i=p_2^\prime}^\infty$.
    \item[(d)] $ [(t_i(x))_{i=p_1}^{p_2-1} \; \| \;
  (A_{k+i})_{i=p_1}^{p_2-1}]_v \subseteq
  [(t_i^\prime(x))_{i=p_1'}^{p_2'-1}
  \; \| \; (A_{k'+i})_{i=p_1'}^{p_2'-1}]_v.$
\end{enumerate}
\end{sublem}
\begin{proof}
Using Definition
\ref{block} and assumption (ii)  we find a unique  positive integer
$d$
  such that
\begin{equation}\label{equat1}
  t_{p_1}(x)...t_{p_2-1}(x) \in [(t^\prime _i(x))_{i=p_1 ^\prime}
  ^{p_1'+d-1} \; \| \; (A_{k^\prime  +i})
  _{i=p^\prime _1}^{p^\prime _1 +d-1}]_v
  \end{equation} and
  \begin{equation}\label{equat1b}  (t_i(x))_{i=p_2}^\infty \unlhd_{k^\prime +
  p_1
    ^\prime+d } (t^\prime _i(x))_{i=p_1'+d}^\infty.\end{equation}
By  \eqref{equat1} we easily conclude that
\begin{equation}\label{equat1c}
p_2-p_1\leq d.
\end{equation}
Moreover, since the sequence of the alphabets $(A_n)_n$ is
increasing,  by \eqref{equat1} and assumption (i),  we get that
\begin{equation}\label{equat1d}
[(t_i(x))_{i=p_1}^{p_2-1} \; \| \;
  (A_{k+i})_{i=p_1}^{p_2-1}]_v \subseteq
  [(t_i^\prime(x))_{i=p_1'}^{p_1'+d-1}
  \; \| \; (A_{k'+i})_{i=p_1'}^{p_1'+d-1}]_v.
\end{equation}
We set $p_2'=p_1'+d$. Then (a) follows from  \eqref{equat1c}. Also,
using  (i), we have that
    $$k+p_2 = k+p_1 +(p_2-p_1)  \leq
    k^\prime +p^\prime _1 + d = k^\prime +p^\prime _2$$ and so  (b) is 
     satisfied too.
    Finally, using \eqref{equat1b} and \eqref{equat1d} we easily 
    derive  (c) and
    (d) respectively.
\end{proof}

\begin{proof}[Proof of Lemma \ref{Cf9}.]  By induction
  for every $n \geq 1$ we select a finite sequence $(p_n ^l)_{l=0}^n$
in $ \nn$ with $p_n^n =0$ and  $p_{n}^{n-1} =m _{n}$
satisfying for every $n \geq 1$ the following conditions.

\begin{enumerate}

\item[(a)] $p_{n-1}^l<p_{n}^l$ for all $0\leq l \leq n-1$ and for every
$1 \leq l \leq n-1$ we have $$p_n^l - p_{n-1}^l \leq p_n^{l-1} -
p_{n-1}^{l-1}.$$

\item[(b)] For every   $1 \leq l \leq n$ we have
$$k_{l}+p_n
^{l}\leq k_{l-1}+p_n^{l-1}.$$

\item[(c)] For every   $1 \leq l \leq n$ we have
$$(t_i^{(l)}(x))_{i=p_n^{l}}^\infty \unlhd _{k_{l-1} + p_n^{l-1}}
(t_i^{(l-1)}(x))_{i=p_n^{l-1}}^\infty.$$

\item[(d)] For every   $1 \leq l \leq n-1$ we have
$$ [(t_i^{(l)}(x))_{i=p_{n-1}^{l}}^{p_{n}^l-1} \; \| \;
  (A_{k_{l}+i})_{i=p_{n-1}^{l}}^{p_{n}^{l}-1}]_v \subseteq
  [(t_i^{(l-1)}(x))_{i=p_{n-1}^{l-1}}^{p_{n}^{l-1}-1}
  \; \| \; (A_{k_{l-1}+i})_{i=p_{n-1}^{l-1}}^{p_{n}^{l-1}-1}]_v.$$
\end{enumerate}
The above inductive construction is easily carried out using
Sublemma \ref{extralemma}.

We set $p_0 =0$ and $p_n  =p_n^0$, for all $n \geq 1$. We claim
that the sequence $(p_n)_{n=0}^ \infty$ is as desired. Indeed, by
(a) we have $p_{n-1}^0<p_{n}^0$ for all $n \geq 1$ and hence the
sequence $(p_n)_{n=0}^ \infty$ is strictly increasing. Moreover,
by (b) we easily get (C1). Finally,  by (ii) and (d) of Lemma \ref{Cf9},
it follows
that for every $n\in \nn$ and every $0 \leq l \leq n$ we have
\begin{equation} \label{eq1} w_{n}(x) \in [(
  t_i^{(l)}(x))_{i=p_{n}^l}^{p_{n+1}^l-1}\; \|\;
  (A_{k_l+i})_{i=p_{n}^l}^{p_{n+1}^l-1} ]_v.\end{equation}
Setting $l=0$ in \eqref{eq1} we derive (C2). To verify (C3) we fix
$n_0\in\nn$. Setting $l=n_0$ in \eqref{eq1} we get that
$(w_n(x))_{n=n_0}^\infty $ is a reduced $k_{n_0}$-block
subsequence of $ \vt_{n_0}$ and the proof of the lemma is
complete.
\end{proof}

\subsubsection{Shifting sequences of variable words.}
 Let $s(x)$ be a variable word over $A$. By $s(x)^*$ we denote
the maximal initial segment of $s(x)$ which is a constant word and
by $s(x)^{**}$ the maximal final segment of $s(x)$ for which is a
left variable word.

Clearly, $ s(x)=s(x)^*s(x)^{**}$ for every $s(x)\in V(A)$.
Moreover, if  $s(x)$ is a left variable word then
 $s(x)^*$ is the empty word.

\begin{defn} We define a map $S: \V \rightarrow \V$  as follows.
For every
 $\vs =(s_n(x))_{n=0}^ \infty\in \V$ we set
$S(\vs)=(w_n(x))_{n=0}^ \infty$, where $w_0(x)=s_0(x)s_1(x)^*$  and for every
$n \geq 1$, $w_n(x)=s_n(x)^{**} s_{n+1}(x)^*$.
\end{defn}
Notice that for every $\vs\in \V$, $S(\vs)$ is a sequence of
variable words which all  except perhaps the first, are left
variable words. In addition, for every  $\vs=(s_n(x))_{n=0}^ \infty\in \V$
such that $s_n(x)$ is a left variable word for all $n\geq 1$, we
have that  $S(\vs)=\vs$.

It is easy to see that the map  $S$ preserves the infinite reduced
$k$-block subsequences. More precisely, we have the following.

\begin{fact}\label{Cf11}
  Let $k\in \nn$ and
     $\vs, \vt \in \V$. If $\vt \unlhd_k \vs$ then $S(\vt)\unlhd_k
S(\vs)$.
\end{fact}

\subsubsection{The notion of $(S,k)$-large families}
The next definition is crucial for the proof of Theorem
\ref{thm2}.
\begin{defn}\label{defn large}
  Let $k \in \nn$, $E \subseteq W(A)$  and $\vs\in\V$.
  Then $E$ will be called
  $(S,k)$-\textit{large} in $\vs$ if $$E \cap [S( \vw ) \  \| \
  (A_{k+n})_{n=0}^ \infty]_c \neq \emptyset,$$
  for every infinite reduced $k$-block subsequence
  $\vw$ of $\vs$.\end{defn}
In other words  $E$ is $(S,k)$-large in $\vs$ if for every
infinite reduced $k$-block subsequence  $\vw=(w_n(x))_{n=0}^ \infty$ of $\vs$ there
  exist $m \in \nn$ and  $a_i\in A_{k+i}$ for every $0\leq i\leq m$ such that
  $$w_0(a_0) ... w_{m}(a_{m}) w_{m+1}(x)^* \in E.$$
  By Fact \ref{Cf8} we easily obtain the following.
\begin{fact}\label{Cf13}
   Let $k \in \nn$, $E \subseteq W(A)$ and
  $\vec{s} \in \V$
   such that $E$ is $(S,k)$-large in $\vs$. Then
    for every infinite reduced $k$-block subsequence
$\vt$ of  $\vs$ we have that $E$ is $(S,k)$-large in $\vt$.
\end{fact}

\begin{lem}\label{Cl14}
   Let $k \in \nn$, $E \subseteq W(A)$ and
  $\vec{s} \in \V$
   such that $E$ is $(S,k)$-large in $\vs$. Let $r\geq 2$   and let
   $E=\bigcup_{i=1}^r
   E_i$.
    Then there exist $1\leq i \leq r$ and an
   infinite reduced $k$-block subsequence $\vt$
   of $\vs$ such that $E_i$ is $(S,k)$-large in $\vt$.
\end{lem}

\begin{proof}
We first show the result for   $r=2$.
  Assume to the contrary that for every  infinite reduced $k$-block
  subsequence $\vt$
   of $\vs$ neither  $E_1$ nor $E_2$ is $(S,k)$-large in $\vt$.

   Then $E_1$ is not
   $(S,k)$-large in $\vs$ and so   there exists
  $\vt _1 \unlhd_k \vs$ such that $$ E_1 \cap [S( \vt_1 ) \  \| \
  (A_{k+n})_{n=0}^ \infty]_c  =
  \emptyset.$$
  Similarly,   $E_2$ is not $(S,k)$-large in $\vt_1$ and so  there
  exists
  $\vt _2 \unlhd_k \vt_1$ such that $$ E_2 \cap [S( \vt_2 ) \  \| \
  (A_{k+n})_{n=0}^ \infty]_c  =
  \emptyset.$$
  Since $\vt _2 \unlhd_k \vt_1$, by   Fact \ref{Cf8} we have that
  $\vt _2 \unlhd_k \vs$.
  Moreover, by   Fact \ref{Cf11} we get that $S(\vt _2) \unlhd_k S(\vt_1)$
  and therefore by  Fact \ref{Cf7}
  we have that
$$[S( \vt_2 ) \  \| \ (A_{k+n})_{n=0}^ \infty]_c\subseteq [S( \vt_1 ) \  \| \
(A_{k+n})_{n=0}^ \infty]_c.$$ Hence, $E \cap [S( \vt_2 ) \ \| \
(A_{k+n})_{n=0}^ \infty]_c= \emptyset$,
    a contradiction
   since  $E$ is $(S,k)$-large in $\vs$  and  $t_2 \unlhd_k \vs$.
   Hence, the lemma holds true for $r=2$. Using  induction
   the result follows.
\end{proof}

\subsection{The main arguments}
We pass now to the core of the proof. We will need  the next
definition.
\begin{defn}  Let  $E$ and  $ F$ be non empty subsets of  $W(A)$. We
define
  \[E_F=\{z\in W(A):w z \in E \text{ for every }w\in
  F\}.\]
\end{defn}

Notice that for every $F_1, F_2\subseteq W(A)$ we have
     $(E_{F_1})_{_{F_2}}= E_{F_1  F_2}$, where $F_1  F_2 =
     \{w_1 w_2 : w_1 \in F_1, w_2 \in F_2\}$.

The next lemma is the first main step towards the proof of Theorem
\ref{thm2} and it  is the point where we use the Hales--Jewett
theorem.

\begin{lem}\label{Cl16}
   Let $k \in \nn$, $E \subseteq W(A)$ and $\vs=(s_n(x))_{n=0}^ \infty \in \V$
    such that $E$ is $(S,k)$-large in $\vs$. Then there exist
    $ m\geq1$,
    $w(x) \in [(s_n(x))_{n=0}^{m-1}\ \| \   (A_{k+n})_{n=0}^{m-1}]_v$
   and $\vt \in \V$ with $\vt \unlhd _{k+m} (s_{n}(x))_{n=m}^\infty$ such that
   setting $F=\{w(a):a \in A_{k}\}$ then
   $E_F$ is $(S,(k+m))$-large in $\vt$.
\end{lem}

\begin{proof} Fix  $k \in \nn$, $E \subseteq W(A)$ and
$\vs=(s_n(x))_{n=0}^ \infty \in \V$
and assume  that $E$ is $(S,k)$-large in $\vs$.

\begin{claim}\label{Claim17}
    There exist $r \in \nn$ and a finite sequence
   $(w_i(x))_{i=0}^r$ with
    $(w_i(x))_{i=0}^r \unlhd _k \vs$ such that for every  $w(x)\in V(A)$
    with  $(w_0(x),..., w_r(x),w(x)) \unlhd _k
    \vs$
   there exists $(a_0,...,a_r) \in \prod_{n=0}^r A_{k+n}$ such that
   $w_0(a_0) ... w_r(a_r) w(x)^* \in E.$
\end{claim}

 \begin{proof}[Proof of Claim \ref{Claim17}]  Assume that the conclusion fails.
 By induction we easily  construct  a sequence of variable words
  $\vw =(w_n(x))_{n=0}^ \infty \in \V$ with $w_0(x)=s_0(x)$,
  $\vw \unlhd_k\vs$ and
 $[S(\vw)\;\|\; (A_{k+n})_{n=0}^ \infty]_c \subseteq E^c$, a contradiction.
\end{proof}

\begin{claim}\label{Claim18} Let  $(w_i(x))_{i=0}^r$ be as  in Claim \ref{Claim17}. Let
$n_0 \geq 1$ be the unique integer such that $w_0(x) ... w_r(x)
\in [(s_n(x))_{n=0}^{n_0-1}\;\|\; (A_{k+n})_{n=0}^{n_0-1}]_v$ and
set
\[ \ q=\prod_{n=0}^{r}|A_{k+n}|, \ N=HJ(|A_k|,q) \ \text{and}
   \ m=n_0+N,\]
   where $HJ(|A_k|,q)$ is as in Theorem \ref{thm1}.

Then for every variable word
 $v(x) \in [(s_n(x))_{n=m}^\infty \;\|\;(A_{k+n})_{n=m}^\infty]_v$ there
 exist a constant word
  $$w\in [(s_n(x))_{n=0}^{n_0-1}\;\|\;
 (A_{k+n})_{n=0}^{n_0-1}]_c$$  and  a variable word
$$y(x) \in [(s_n(x))_{n=n_0}^{n_0+N-1}\;\| \; A_{k}]_v$$
 such that
    $ w y(b)  v(x)^* \in E$  for all  $b \in A_k$.
\end{claim}

 \begin{proof}[Proof of Claim \ref{Claim18}] We set $B= A_k$ and we fix
  $v(x) \in [(s_n(x))_{n=m}^\infty
 \;\|\;(A_{k+n})_{n=m}^\infty]_v$. We will  define a  finite coloring of
 $B^N$, depending on  $v(x)$,
 as follows.
Let $\mathbf{b}=(b_0,...,b_{N-1}) \in
 B^N$ be arbitrary and let
  $z(x)= s_{n_0}(b_0) ...
   s_{n_0+N-1}(b_{N-1}) v(x).$
Since $B=A_k\subseteq A_{k+n}$ for all $n\in\nn$, we have that
$z(x) \in [(s_n(x))_{n=n_0}^\infty
\;\|\;(A_{k+n})_{n=n_0}^\infty]_v.$
  Therefore $(w_0(x),..., w_r(x),z(x))$ is a finite reduced
   $k$-block subsequence of $\vs$. Hence, by Claim \ref{Claim17}
   we may select  $(a_0^{\mathbf{b}},...,a_r^{\mathbf{b}}) \in \prod_{n=0}^{r} A_{k+n}$
    such that
   $ w_0(a_0^{\mathbf{b}}) ... w_r(a_r^{\mathbf{b}}) z(x)^*$
   $\in E$.
We define  $$c_{v(x)}:B^N\to\prod_{n=0}^{r} A_{k+n}$$ by setting
$c_{v(x)}(\mathbf{b})=(a_0^{\mathbf{b}},...,a_r^{\mathbf{b}})$.

Since $q=\prod_{n=0}^{r} |A_{k+n}|$ and   $N=HJ(|B|,q)$, by the
Hales-Jewett Theorem and the way we define $c_{v(x)}$, we conclude
that  there exists a variable word $$ y(x) \in
[(s_n(x))_{n=n_0}^{n_0+N-1}\;\| \; B]_v$$ and an $(r+1)$-tuple
$(a_0,...,a_r) \in \prod_{n=0}^{r} A_{k+n}$ such that $$ w_0(a_0)
... w_r(a_r)y(b) v(x)^* \in E,$$ for all $b \in B$.

We set $w= w_0(a_0) ... w_r(a_r)$. Since $w_0(x)  ... w_r(x) \in
   [(s_n(x))_{n=0}^{n_0-1}\;\|\; (A_{k+n})_{n=0}^{n_0-1}]_v$ we
   easily see that  $w\in[(s_n(x))_{n=0}^{n_0-1}\;\|\;
   (A_{k+n})_{n=0}^{n_0-1}]_c$ and  $wy(b)v(x)^*\in
   E$, for all $b\in B =A_k$.
   \end{proof}

\begin{claim} \label{Claim119} Let  $(w_i(x))_{i=0}^r$ be as  in
Claim \ref{Claim17} and
$n_0$,  $N$ and  $m$  be as in Claim \ref{Claim18}. Let
$$\mathbb{P}=[(w_n(x))_{n=0}^r\;\|\; (A_{k+n})_{n=0}^r]_c \times
[(s_n(x))_{n=n_0}^{n_0+N-1}\;\| \; A_{k}]_v$$
and for every $(w,y(x))\in\mathbb{P}$ let $F_{(w,y(x))}=\{w y(a) :
a \in A_k\}$. Finally, let
\[E^*=\bigcup_{(w,y(x))\in\mathbb{P}} E_{F_{(w,y(x))}},\]
where $E_{F_{(w,y(x))}}=\{z\in W(A): uz \in E \text{ for all }
u\in F_{(w,y(x))}\}$.

Then  $E^*$ is $(S,(k+m))$-large in
 $(s_n(x))_{n=m}^\infty$.
 \end{claim}

\begin{proof}[Proof of Claim \ref{Claim119}]
Let $\vt=(t_n(x))_{n=0}^ \infty$ be an arbitrary reduced $(k+m)$-block
subsequence  of
$(s_n(x))_{n=m}^\infty$. Let  $a_0 \in A_{k+m}$ and set
$v(x)=t_0(a_0) t_1(x).$ Clearly, $$v(x) \in [(s_n(x))_{n=m}^\infty
\;\|\; (A_{k+n})_{n=m}^\infty]_v$$ and so, by Claim \ref{Claim18},
there exists a pair $(w,y(x))\in\mathbb{P}$ such that
$wy(b)v(x)^*\in E$, for all $b\in A_k$ or equivalently
$v(x)^*=t_0(a_0) t_1(x)^* \in  E_{F_{(w,y(x))}}\subseteq E^*$.
Therefore, $E^*$ is $(S,(k+m))$-large in
 $(s_n(x))_{n=m}^\infty$.
\end{proof}

We are now ready to finish the proof of the lemma. Indeed, by
Claim \ref{Claim119} and  Lemma \ref{Cl14}, there exist a pair
$(w_0,y_0(x))\in\mathbb{P}$ and a reduced $(k+m)$-block subsequence $\vt$ of
$(s_n(x))_{n=m}^\infty$ such that $E_{F_{(w_0,y_0(x))}}$ is
$(S,(k+m))$-large in $\vt$. We set $$w(x) =w_0 y_0(x) \text{ and }
F=F_{(w_0,y_0(x))}=\{w(a): a\in A_k\}.$$ Then $w(x) \in
[(s_n(x))_{n=0}^{m-1}\;\|\; (A_{k+n})_{n=0}^{m-1}]_v$ and $E_F$ is
$(S,(k+m))$-large in $\vt$, as desired.
\end{proof}

\begin{lem}\label{Cl20}  Let $k \in \nn$, $E \subseteq W(A)$ and $\vs=
(s_n(x))_{n=0}^ \infty \in \V$ such that
   $E$ is $(S,k)$-large in $\vs$. Then there exist a sequence
   $(w_n(x))_{n=0}^ \infty$ of variable words and two
strictly increasing sequences $(k_n)_{n=0}^ \infty$ and
$(p_n)_{n=0}^ \infty$ in $ \nn $,
with $k_0=k$ and $p_0=0$ such that setting for every $n \in \nn$,
$F_n=[ (w_i(x))_{i=0}^n\;\|\; (A_{k_i})_{i=0}^{n}]_c$ then
for every $n \in \nn$ the
following properties  are satisfied.
\begin{enumerate}
\item[(P1)] $k+p_n \geq k_n$. \item[(P2)]  $w_n(x) \in
[(s_i(x))_{i=p_n}^{p_{n+1}-1} \; \| \;
(A_{k+i})_{i=p_n}^{p_{n+1}-1}]_v$. \item[(P3)]  $E_{F_n}$ is
$(S,k_{n+1})$-large in $(w_{i}(x))_{i=n+1}^\infty$.
\end{enumerate}
\end{lem}

\begin{proof}
 By repeated use of Lemma \ref{Cl16} we obtain
   a sequence $(w_n(x))_{n=0}^ \infty$ of variable words,  a
sequence $(\vt_n)_{n=0}^ \infty$ in  $ \V$ with $\vt_0 =\vs$, and two
sequences $(k_n)_{n=0}^ \infty$ and  $(m_n)_{n=0}^ \infty$ in
$ \nn $, with $k_0=k$ and
$m_0=0$, such that setting for every $n\in\nn$,
$\vt_n=(t_i^{(n)}(x))_{i=0}^\infty$
then for every $n\geq 1$ the following  are satisfied.
\begin{enumerate}
\item[(i) ]  $m_{n}\geq 1$ and $k_{n}=k_{n-1}+m_{n}$.

\item[(ii)]
 $w_{n-1}(x) \in [(t^{(n-1)}_i(x))_{i=0}^{m_{n}-1}\ \| \
(A_{k_{n-1}+i})_{i=0}^{m_{n}-1}]_v$.

\item[(iii)] $\vt _{n}$ is a reduced $k_{n}$-block subsequence of
 $(t_{i}^{(n-1)}(x))_{i=m_{n}}^\infty$.

\item[(iv)]  The set $E_{F_{n-1}}$ is $ (S,k_{n})\text{-large in }
\vt_{n}$.
\end{enumerate}
Assertions (i) - (iii) allow us to apply  Lemma \ref{Cf9} and
obtain a strictly increasing sequence $(p_n)_{n=0}^ \infty$ with $p_0=0$
satisfying (C1)-(C3) of that lemma. Notice that (C1) and (C2) are
identical to (P1) and (P2) respectively. By (C3) we have that
$(w_{i}(x))_{i=n+1}^\infty$ is a  reduced $k_{n+1}$-block subsequence of
$\vt_{n+1}$ and by (iv) above we have that $E_{F_{n}}$ is $
(S,k_{n+1})\text{-large in } \vt_{n+1}$. Hence, by Fact \ref{Cf13}
we have that (P3) is also satisfied.
\end{proof}

\begin{lem}\label{Cl21}
 Let $\vec{w}=(w_n(x))_{n=0}^ \infty$, $(k_n)_{n=0}^ \infty$  and
 $(p_n)_{n=0}^ \infty$ be the sequences obtained in Lemma
 \ref{Cl20}. Also let $F_n=[ (w_i(x))_{i=0}^n\;\|\;
 (A_{k_i})_{i=0}^{n}]_c,$ for all $n \in \nn$.
Then there exist a strictly increasing sequence
 $(r_n)_{n=0}^ \infty$ in $\nn$ with $r_0 =0$
 and a sequence $\vt=(t_n(x))_{n=0}^ \infty$  of variable words  such that
for every  $n\geq 1$ the following are
  satisfied.
 \begin{enumerate}

 \item[(Q1)] $t_{n}(x) \in
 [(w_{r_n+i}(x))_{i=0}^{r_{n+1}-r_{n}-1}
  \; \|\; (A_{k_{r_{n}}+i})_{i=0}^{r_{n+1}-r_{n}-1}]_v$.

 \item[(Q2)] For every $u\in [(t_i(x))_{i=0}^{n-1}   \; \|\;
 (A_{k_{r_{i+1}-1}})_{i=0}^{n-1}]_c$ we have $u t_{n}(x) ^* \in
 E$.

  \item[(Q3)] $[(t_i(x))_{i=0}^n \; \| \; (A_{k_{r_{i+1}-1}})_{i=0}^n]_c \subseteq
  F_{r_{n+1}-1}$.

 \end{enumerate}
\end{lem}

 \begin{proof} We proceed by induction to the construction
 of the desired sequences.
 For $n=1$ we set $r_0=0$, $r_1=1$ and  $t_0(x)=w_0(x)$. Let
   $$G_0=[t_0(x)\; \|\; A_{k_{r_1-1}}]_c.$$ Then
   $G_0 =F_0$ and therefore by Lemma \ref{Cl20} we have that $E_{G_0}$ is
$(S,k_{1})$-large in $(w_{i}(x))_{i=1}^\infty$. Hence, there exist
$r \in\nn$ and $a_i \in A_{k_1+i}$, $0\leq i\leq r$ such that for
every $u\in G_0$ we have
$$ u w_{r_1}(a_0)...w_{r_1+r}(a_{r})w_{r_1+r+1}(x)^*
\in E.$$
 We set $r_2=r_1+r+2$ and $$t_1(x)=w_{r_1}(a_0)...w_{r_1+r}(a_{r})w_{r_1+r+1}(x)=
 w_{r_1}(a_0)...w_{r_2-2}(a_{r})w_{r_2-1}(x).$$
 Notice that $t_1(x)^*=w_{r_1}(a_0)...w_{r_2-2}(a_{r})w_{r_2-1}(x)^*$ and
  $$ [(t_i(x))_{i=0}^1\; \|\;
   (A_{k_{r_{i+1}-1}})_{i=0}^1]_c\subseteq
   [ (w_i(x))_{i=0}^{r_2-1}\;\|\; (A_{k_i})_{i=0}^{r_2-1}]_c= F_{r_2-1}.$$
By the above we have that  (Q1) - (Q3) are satisfied for $n=1$.

 Assume that the construction has been carried out up to some $n \geq 1$.
 We set $$G_n =[(t_i(x))_{i=0}^n \; \| \; (A_{k_{r_{i+1}-1}})_{i=0}^n]_c$$ and
 by our inductive hypothesis we have that $G_n \subseteq
  F_{r_{n+1}-1}$. Therefore by Lemma \ref{Cl20} we have that $E_{G_n}$
  is $(S,k_{r_{n+1}})$-large in $(w_{i}(x))_{i=r_{n+1}}^\infty$.
Hence,  there exist $r \in\nn$ and $a_i \in A_{k_{r_{n+1}}+i}$,
$0\leq i\leq r$ such that for every $u\in G_n$ we have
$$ u w_{r_{n+1}}(a_0)...w_{r_{n+1}+r}(a_{r})w_{r_{n+1}+r+1}(x)^*
\in E.$$
 We set $r_{n+2}=r_{n+1}+r+2$ and \[\begin{split}t_{n+1}(x) & =
 w_{r_{n+1}}(a_0)...w_{r_{n+1}+r}(a_{r})w_{r_{n+1}+r+1}(x)\\ &
 =w_{r_{n+1}}(a_0)...w_{r_{n+2}-2}(a_{r})w_{r_{n+2}-1}(x).\end{split}\]
 Notice that
 $t_{n+1}(x)^*=w_{r_{n+1}}(a_0)...w_{r_{n+2}-2}(a_{r})w_{r_{n+2}-1}(x)^*$
 and
  $$[(t_i(x))_{i=0}^{n+1} \; \| \; (A_{k_{r_{i+1}-1}})_{i=0}^{n+1}]_c\subseteq
   [ (w_i(x))_{i=0}^{r_{n+2}-1}\;\|\; (A_{k_i})_{i=0}^{r_{n+2}-1}]_c= F_{r_{n+2}-1}.$$
   It is
 easily checked that  (Q1) - (Q3) are satisfied and the proof of the inductive step is complete.
\end{proof}

\begin{cor}\label{corol22}
   Let $k \in \nn$, $E \subseteq W(A)$ and $\vs \in \V$ such that
   $E$ is $(S,k)$-large in $\vs$.
   Then there exists a reduced
   $k$-block subsequence $\vt$ of $\vs$
   such that
    $$[S(\vt)\;\|\; ( A_{k+n})_{n=0}^ \infty]_c \subseteq E.$$
\end{cor}

\begin{proof}
Let $(r_n)_{n=0}^ \infty$ and $\vt= (t_n(x))_{n=0}^ \infty$ be the
sequences obtained in
Lemma \ref{Cl21}. Fix $j\in\nn$. By (P2) of Lemma \ref{Cl20} we
have that
 $$w_{r_n+j}(x) \in
[(s_i(x))_{i=p_{r_n+j}}^{p_{{r_n+j}+1}-1} \; \| \;
(A_{k+i})_{i=p_{r_n+j}}^{p_{{r_n+j}+1}-1}]_v.$$ Moreover by (P1)
of Lemma \ref{Cl20} we have that $$k+p_{r_n+j} \geq k_{r_n+j}.$$
Hence, for all $j \in \nn$, $$[w_{r_n+j}(x) \; \| \;
A_{k_{r_n}+j}]_c \subseteq
[(s_i(x))_{i=p_{r_n+j}}^{p_{{r_n+j}+1}-1} \; \| \;
(A_{k+i})_{i=p_{r_n+j}}^{p_{{r_n+j}+1}-1}]_c.$$
Therefore we obtain that
$$[(w_{r_n+j}(x))_{j=0}^{r_{n+1}-r_n-1}\; \|\;
(A_{k_{r_n}+j})_{j=0}^{r_{n+1}-r_n-1}]_c
 \subseteq
  [(s_i(x))_{i=p_{r_n}}^{p_{r_{n+1}}-1} \;
   \| \; (A_{k+i})_{i=p_{r_n}}^{p_{r_{n+1}}-1}]_c$$
  and so
  $$[(w_{r_n+j}(x))_{j=0}^{r_{n+1}-r_n-1}\; \|\;
(A_{k_{r_n}+j})_{j=0}^{r_{n+1}-r_n-1}]_v
 \subseteq
  [(s_i(x))_{i=p_{r_n}}^{p_{r_{n+1}}-1} \;
   \| \; (A_{k+i})_{i=p_{r_n}}^{p_{r_{n+1}}-1}]_v$$
   Hence, by (Q1) of Lemma \ref{Cl21} we have that  $$t_n(x) \in
   [(s_i(x))_{i=p_{r_n}}^{p_{r_{n+1}}-1} \;
   \| \; (A_{k+i})_{i=p_{r_n}}^{p_{r_{n+1}}-1}]_v,$$
 i.e. $\vt$ is a reduced $k$-block subsequence of $\vs$.

 Recall that
$S(\vt) =(t_0(x) t_1(x)^*,t_1(x)^{**}t_2(x)^*,...).$ We set
$u_0(x) = t_0(x)t_1(x)^*$ and for all $n \geq 1 $ we set
$u_n(x)=t_n(x)^{**}t_{n+1}(x)^*$. Then $$[S(\vt) \; \| \;
(A_{k+n})_{n=0}^ \infty]_c = [(u_n(x))_{n=0}^ \infty \; \| \;
(A_{k+n})_{n=0}^ \infty]_c.$$

Let $n \in \nn$ and for every $0 \leq i \leq n$ let $a_i \in
A_{k+i}$. Notice that  $$u_0(a_0)
u_1(a_1)...u_n(a_n)=t_0(a_0)...t_n(a_n)t_{n+1}(x)^*.$$ We set
$u=t_0(a_0) t_1(a_1)...t_n(a_n)$.
 Since the
sequences $(k_n)_{n=0}^ \infty$ and $(r_n)_{n=0}^ \infty$  are
increasing we obtain that
$$k_{r_{i+1}-1} \geq k_{i} \geq k+i.$$ Therefore,
$$u \in [(t_i(x))_{i=0}^n \; \| \;
(A_{k+i})_{i=0}^n]_c \subseteq [(t_i(x))_{i=0}^{n}   \; \|\;
 (A_{k_{r_{i+1}-1}})_{i=0}^{n}]_c.$$
Hence by (Q2) of Lemma \ref{Cl21}  we have that
$$ut_{n+1}(x)^*=t_0(a_0)...t_n(a_n)t_{n+1}(x)^* \in E.$$
Therefore  $[S(\vt)\;\|\; ( A_{k+n})_{n=0}^ \infty]_c \subseteq E$ and the
proof is complete.
\end{proof}

\begin{thm}\label{thm23}
  Let $k \in \nn$ and  $\vec{v}=(v_n(x))_{n=0}^ \infty$ be
  a sequence of variable words such that
  for every $n\geq 1$ $v_n(x)$ is a left variable word.
  Let $r\geq 2$ and $[\vec{v}\;\|\; (A_{k+n})_{n=0}^ \infty]_c = \cup_{i=1}^r
  E_i$. Then
   there exist $1 \leq i\leq r$ and
   a reduced $k$-block subsequence $\vec{w}=(w_n(x))_{n=0}^ \infty$ of $\vec{v}$
   such that for every $n\geq 1$, $w_n(x)$ is
   a left variable word and
   $[\vec{w}\;\|\; (A_{k+n})_{n=0}^ \infty]_c \subseteq E_i$.
\end{thm}

\begin{proof} Since  for every $n\geq
1$, $v_n(x)$ is a left variable word he have that
$S(\vec{v})=\vec{v}$ and so $[S(\vec{v})\;\|\; (A_{k+n})_{n=0}^ \infty]_c =
[\vec{v}\;\|\; (A_{k+n})_{n=0}^ \infty]_c =\cup_{i=1}^r
  E_i.$
Let $\vt \in \V$ such that
      $\vt \unlhd_k \vec{v}$. By Fact \ref{Cf11}, we have that
      $S(\vt) \unlhd_k S(\vec{v})$
and by Fact \ref{Cf7} we obtain that $[S(\vt) \  \| \
(A_{k+n})_{n=0}^ \infty]_c
    \subseteq [S(\vec{v}) \  \| \ (A_{k+n})_{n=0}^ \infty]_c$ and
    therefore $[S(\vt) \  \| \ (A_{k+n})_{n=0}^ \infty]_c
    \subseteq \cup_{i=1}^r
  E_i$. Thus,  trivially, we get that    $\cup_{i=1}^r E_i$ is $(S,k)$-large in
  $\vec{v}$.
Hence, by Lemma \ref{Cl14}, there
  exist $1 \leq i\leq r$ and $\vs \unlhd _k \vec{v} $
  such that $E_i$ is $(S,k)$-large in $\vs$ and so
  by Corollary \ref{corol22},
  there exists a
   $\vt \unlhd _k \vs$
   such that
    \[\label{oi}[S(\vt)\;\|\; ( A_{k+n})_{n}]_c \subseteq
    E_i.\]
    We set $\vw= S(\vt)$. Since $\vt \unlhd _k \vs\unlhd _k \vec{v}$
     we have  $\vw=S(\vt) \unlhd _k S(\vec{v})=\vec{v}$. Hence,
     $\vw\unlhd _k \vec{v}$ and
 $[\vw\;\|\; ( A_{k+n})_{n}]_c \subseteq
    E_i.$ The
     proof is complete.
\end{proof}

\begin{proof}[Proof of Theorem \ref{thm2}] Let $r \geq 2 $ and let $W(A) =
\cup _{i=1}^r E_i$. Let $\vec{v}=(x,x,...)$. Then
 $[\vec{v}\;\|\; (A_{n})_{n=0}^ \infty]_c = W(A)=\cup_{i=1}^r E_i$. Applying
 Theorem \ref{thm23} for $k=0$, we obtain $1 \leq i\leq r$ and
   $\vw=(w_n(x))_{n=0}^ \infty\in \V$
    such that for every $n\geq 1$, $w_n(x)$ is
   a left variable word and
   $[\vec{w}\;\|\; (A_{n})_{n=0}^ \infty]_c \subseteq E_i$.
\end{proof}

\section{Proof of theorem \ref{thm3}}

\subsection{Preliminaries.} \label{Initializing the
proof2}
As in Section \ref{Initializing the proof} we
fix for the following an increasing sequence
$$A_0\subseteq A_1 \subseteq ... \subseteq A_n \subseteq ...$$
 of finite alphabets and we set $$A=\cup_{n\in \nn} A_n.$$ Again, by
 $V(A)$ we denote the set of all variable words (over
 $A$).
 Also by $V^ {<\infty}(A) $ (resp. $V^ \infty(A)
$) we denote the set of all finite (resp. infinite) sequences of
variable words and let $V^{\leq\infty}(A)=V^ {<\infty}(A) \cup
V^ \infty(A)$.

\subsubsection{Extracted constant and variable span of a sequence of variable words.}
 Let  $m\in \nn$,
$(s_n(x))_{n=0}^m\in V^{<\infty}(A)$ and $(k_n)_{n=0}^m$ be a
strictly increasing finite sequence of non negative integers.
 The \textit{extracted constant span of}
$(s_n(x))_{n=0}^m$ \textit{with respect to} $(A_{k_n})_{n=0}^m$
denoted by $<(s_n(x))_{n=0}^m \  \| \ (A_{k_n})_{n=0}^m>_c$
is defined to be the set
$$\bigcup_{n=0}^m\{s_{l_0}(a_0)... s_{l_n}(a_n):
0 \leq l_0<...<l_n\leq m,\;
  a_i\in A_{k_{l_i}}  \text{for all}
\  0\leq i \leq n \}.$$
We also define the \textit{extracted variable span of}
$(s_n(x))_{n=0}^m$ \textit{with respect to} $(A_{k_n})_{n=0}^m$
 to
be the set
\[<(s_n(x))_{n=0}^m \  \| \ (A_{k_n})_{n=0}^m>_v =
V(A) \cap <(s_n(x))_{n=0}^m \  \| \ (A_{k_n} \cup
\{x\})_{n=0}^m>_c,\] that is the extracted variable span of
$(s_n(x))_{n=0}^m$ with respect to $(A_{k_n})_{n=0}^m$ consists of
all variable words of the form $s_{l_0}(b_0)s_{l_1}(b_1)...
s_{l_n}(b_n)$
 such that $0 \leq l_0<l_1<...<l_n\leq m$, $b_i \in A_{k_{l_i}} \cup \{x\}$,
for all $0\leq i \leq n$ and for
at least one $i$ we have $b_i=x$.

The above notation extends to infinite sequences of
variable words as follows. Let $(s_n(x))_{n=0}^ \infty  \in \V$
and $(k_n)_{n=0}^ \infty$
be a strictly increasing sequence of non negative integers. Then
the extracted constant span of $(s_n(x))_{n=0}^ \infty$ with
respect to $(A_{k_n})_{n=0}^ \infty$
denoted by $<(s_n(x))_{n=0}^ \infty \  \| \ (A_{k_n})_{n=0}^ \infty>_c$
is defined to be the set
$$\{s_{l_0}(a_0)s_{l_1}(a_1)... s_{l_n}(a_n):
n\in \nn,\ 0\leq l_0<l_1<...<l_n,\; a_i\in A_{k_{l_i}} \text{for all }0\leq
i\leq n \}$$
 and the extracted variable span of $(s_n(x))_{n=0}^ \infty$ with
respect to $(A_{k_n})_{n=0}^ \infty$ is the set
\[<(s_n(x))_{n=0}^ \infty \  \| \ (A_{k_n})_{n=0}^ \infty>_v =
V(A) \cap <(s_n(x))_{n=0}^ \infty \ \|
\ (A_{k_n} \cup \{x\})_{n=0}^ \infty>_c.\]
In the following  we will also write $<\vs \  \| \
(A_{k_n})_{n=0}^ \infty>_c$
(resp. $<\vs \  \| \ (A_{k_n})_{n=0}^ \infty>_v$) to
denote the the extracted
constant (resp. variable) span of $\vs=(s_n(x))_{n=0}^ \infty$
with respect to
$(A_{k_n})_{n=0}^ \infty$.

\subsubsection{Extracted $k$-block subsequences  of a sequence of variable words.}
Next we specify a notion of a  ``block subsequence" of a sequence
$\vs\in\V$ which is related to the extracted span of $\vs$.
\begin{defn}\label{ublock}
Let $k \in \nn$.

\begin{enumerate}
  \item[(i)]Let $l\in \nn$,
   $\vt=(t_n(x))_{n=0}^l\in V^{<\infty}(A)$ and
$\vs=(s_n(x))_{n=0}^ \infty\in \V$. We say that
$\vt$ is a (finite) extracted
$k$-block subsequence of $\vs$ if there exist
$0=m_0<...<m_{l+1}$ such
that
$$t_i(x)\in <(s_n(x))_{n=m_{i}}^{m_{i+1}-1} \ \| \
(A_{k+n})_{n=m_{i}}^{m_{i+1}-1}>_v,$$ for all
$0 \leq i \leq l$.
 \item[(ii)] Let $\vt=(t_n(x))_{n=0}^\infty ,
 \vs=(s_n(x))_{n=0}^\infty \in \V$. We say
that $\vt$ is an (infinite) extracted $k$-block
subsequence of $\vs$ if for
every $l\in \nn$, the finite sequence $(t_n(x))_{n=0}^l$
is a (finite) extracted
$k$-block subsequence of $\vs$.
\end{enumerate}
\end{defn}

In the following we will write  $\vt \leq_k \vs$, whenever $\vt
\in V^{\leq \infty}(A)$, $\vs\in \V$ and $\vt$ is an extracted
$k$-block subsequence of $\vs$.

Notice that if $\vt$ is a finite (resp. infinite) extracted
$k$-block subsequence of $\vs$ then by Definition \ref{ublock}
 there exists a
finite (resp. infinite) sequence of non negative integers
$(m_i)$ such that $t_i(x)\in <(s_n)_{n=m_{i}}^{m_{i+1}-1} \ \| \
(A_{k+n})_{n=m_{i}}^{m_{i+1}-1}>_v$. In contrast to the case of
reduced $k$-block subsequences (see, Definition \ref{block}) this sequence of
non negative integers is not necessarily unique. This is due to
the way extracted spans and block subsequences are defined.

Moreover, taking into account that  the sequence of alphabets
$(A_n)_{n=0}^ \infty$ is increasing, the next facts follow easily from the
 above definitions.
\begin{fact}\label{Cf25}
 Let $k \in \nn$ and $\vs,\vt  \in \V$. If  $\vt \leq _k
\vs$ then   $$<\vt \; \| \; (A_{k+n})_{n=0}^ \infty>_c
\subseteq <\vs\;\|\;
  (A_{k+n})_{n=0}^ \infty>_c$$ and $$<\vt \; \| \;
  (A_{k+n})_{n=0}^ \infty>_v \subseteq <\vs\;\|\;
(A_{k+n})_{n=0}^ \infty>_v.$$
\end{fact}
\begin{fact}\label{Cf26}
 Let $k_0,k_1 \in \nn$, $\vs,\vt, \vw \in \V$. If
 $k_0\leq k_1$ and  $\vw \leq _{k_0} \vt\leq_{k_1} \vs$ then
 $\vw \leq_{k_1} \vs.$
\end{fact}
 The following lemma corresponds to  Lemma \ref{Cf9}. The proof is
 similar.

\begin{lem}\label{Cl27}
 Let $(\vt_n)_{n=0}^ \infty$ be a sequence  in  $ \V$ and
 $(w_n(x))_{n=0}^ \infty$ be a sequence of variable
 words. Also let  $(k_n)_{n=0}^\infty$ and $(m_n)_{n=0}^\infty$
  with $m_0=0$ be two sequences  in $
 \nn$. Let $\vt_n=(t_i^{(n)}(x))_{i=0}^\infty$ for all
 $n\in\nn$ and assume that for every $n \geq 1 $ the
 following are satisfied.
\begin{enumerate}
  \item[(i)] $m_{n} \geq 1$ and $k_{n}=k_{n-1}+m_{n}$.

 \item[(ii)]$w_{n-1}(x) \in <( t_i^{(n-1)}(x))_{i=0}^{m_{n}-1}\; \|\;
  (A_{k_{n-1}+i})_{i=0}^{m_{n}-1} >_v$.

 \item[(iii)]$\vt _{n}$ is an extracted $k_{n}$-block subsequence of
 $(t_{i}^{(n-1)}(x))_{i=m_{n}}^\infty$.
\end{enumerate}
Then there exists a strictly increasing sequence
$(p_n)_{n=0}^ \infty$ in $
\nn$ with $p_0=0$ such that for every $n \in \nn$ the
following are
satisfied.
\begin{enumerate}

\item[(R1)] $k_0+p_n \geq k_n$.

\item[(R2)] $w_n(x) \in <(t_i^{(0)}(x))_{i=p_n}^{p_{n+1}-1} \; \| \;
(A_{k_0+i})_{i=p_n}^{p_{n+1}-1}>_v$.

\item[(R3)] $(w_i(x))_{i=n}^\infty$ is a extracted
$k_n$-block subsequence of
$ \vt_n$.

\end{enumerate}
\end{lem}

\subsubsection{The notion of $k$-large families}
\begin{defn}\label{udefn large}
  Let $k \in \nn$, $E \subseteq W(A)$  and $\vs\in\V$.
  Then $E$ will be called
  $k$-\textit{large} in $\vs$ if $$E \cap <\vw  \  \| \
  (A_{k+n})_{n=0}^ \infty>_c \neq \emptyset,$$
  for every infinite extracted $k$-block subsequence
  $\vw$ of $\vs$.\end{defn}

  By Fact \ref{Cf26} we easily obtain the following.
\begin{fact}\label{ublock fact}
   Let $k \in \nn$, $E \subseteq W(A)$ and
  $\vec{s} \in \V$
   such that $E$ is $k$-large in $\vs$. Then
    for every infinite extracted $k$-block subsequence
$\vt$ of  $\vs$ we have that $E$ is $k$-large in $\vt$.
\end{fact}

Moreover, arguing as in Lemma \ref{Cl14} we obtain the lemma below.

\begin{lem}\label{Cl30}
   Let $k \in \nn$, $E \subseteq W(A)$ and
  $\vec{s} \in \V$
   such that $E$ is $k$-large in $\vs$. Let $r\geq 2$  and
   let $E=\bigcup_{i=1}^r
   E_i$. Then there exist $1\leq i \leq r$ and an
   infinite extracted $k$-block subsequence $\vt$
   of $\vs$ such that $E_i$ is $k$-large in $\vt$.
\end{lem}

We will also need the following.
\begin{lem}\label{exlemma}
Let $k \in \nn$, $E \subseteq W(A)$ and $\vs=(s_n(x))_{n=0}^ \infty \in \V$
such that $E$ is $k$-large in $\vs$. Then for every $m \in \nn$,
$E$ is $(k+m)$-large in $(s_n(x)_{n=m}^\infty$.
\end{lem}
\begin{proof}
  Let $\vt \in \V$ such that $\vt \leq_ {k+m} (s_n(x)_{n=m}^\infty$.
  It is easy to check that  $\vt$ is an extracted $k$-block
  subsequence of $\vs$ and therefore,
  since $E$ is $k$-large in $\vs$, we obtain that
  $$E \cap <\vt  \  \| \ (A_{k+n})_{n=0}^ \infty>_c \neq
  \emptyset.$$ Moreover, since the
  sequence of alphabets $(A_n)$ is increasing we get that
  $$<\vt  \  \| \ (A_{k+n})_{n=0}^ \infty>_c
  \subseteq <\vt  \  \| \
  (A_{k+m+n})_{n=0}^ \infty>_c.$$
  Hence, $E\cap <\vt  \  \| \
  (A_{k+m+n})_{n=0}^ \infty>_c \neq \emptyset$ for every
  $\vt \leq _{k+m}
  (s_n(x))_{n=m}^\infty$, i.e. $E$ is $(k+m)$-large in
  $(s_n(x)_{n=m}^\infty$.
\end{proof}

\subsection{The main arguments}

We remind some notation from Section \ref{Initializing the proof}.
 For a non empty finite subset $B$ of $A$ we set
\[[(s_n(x))_{n=0}^m\;\|\;B]_c=\{s_0(b_0)s_1(b_1)...s_m(b_m):  b_i\in B
 \text{ for all } 0\leq i \leq m \},\] and
\[[(s_n(x))_{n=0}^m\;\|\;B]_v=V(A)\cap [(s_n(x))_{n=0}^m\;\|\;B\cup \{x\}]_c
.\]

\begin{lem}\label{Cl31}
Let $k \in \nn$, $E \subseteq W(A)$ and $\vs=(s_n(x))_{n=0}^ \infty \in \V$
such that $E$ is $k$-large in $\vs$. Then there exist $m \in \nn$
and $w(x) \in <(s_n(x))_{n=0}^m \  \| \ (A_{k+n})_{n=0}^m>_v$ such
that $\{w(a): a\in A_k\}\subseteq  E$.
\end{lem}

\begin{proof}
Assume to the contrary that the conclusion fails. By induction we
construct
  a sequence $\vw = (w_n(x))\leq _k \vs$ such
   that $<\vw \; \| \; (A_{k+n})_{n=0}^ \infty>_c \subseteq E^c$
   which is a contradiction
   since $E$ is $k$-large in
   $\vs$.

   The general inductive step of the construction is as follows.
   Let $N \geq 1$ and assume that $(w_n(x))_{n=0}^{N-1}$
   has been constructed so that
   $$ (w_n(x))_{n=0}^{N-1} \leq _k \vs\text{ and }
   <(w_n(x))_{n=0}^{N-1} \; \| \;
   (A_{k+n})_{n=0}^{N-1}>_c \subseteq
   E^c.$$
Let $n_0 \geq 1$ be the least integer  satisfying $$w_0(x)  ...
w_{N-1}(x) \in
   <(s_n(x))_{n=0}^{n_0-1}\;\|\; (A_{k+n})_{n=0}^{n_0-1}>_v$$
   and let $$q= 2^{ \prod_{i=0}^{N-1}(|A_{k+i}|+1)}.$$ Let
   $H= HJ(|A_{k+N}|,q)$ be as in Theorem \ref{thm1}.
   For every $0 \leq i \leq N-1$  we set $\hat{A} _{k+i} = A_{k+i}
   \cup \{ \emptyset \}$
   and  by convention $w_i( \emptyset) =\emptyset$.
   To each  $w\in   [(s_{n_0+n}(x))_{n=0}^{H-1} \  \| \ A_{k+N}]_c$
   we assign the set of words   $$\{w_0(a_0)...w_{N-1}(a_{N-1})w :
   (a_0,...,a_{N-1}) \in
   \prod_{i=0}^{N-1} \hat{A}_{k+i}\}. $$
   Since $w_0(a_0)...w_{N-1}(a_{N-1})w$ belongs either to $E$ or
   to $E^c$, the above correspondence induces a $q$-coloring on
    the set $ [(s_{n_0+n}(x))_{n=0}^{N-1} \  \| \ A_{k+N}]_c.$
   Therefore, by the Hales--Jewett theorem there exists
   a variable word $$w(x) \in  [(s_{n_0+n}(x))_{n=0}^{N-1} \  \| \ A_{k+N}]_v$$
   such that for every
   $(a_0,...,a_{N-1}) \in
   \prod_{i=0}^{N-1} \hat{A}_{k+i}$  the set
   $$\{ w_0(a_0)...w_{N-1}(a_{N-1}) w(a): a \in
A_{k+N}\}$$  either is included in $E$ or is disjoint from  $E$.
By our initial assumption and since $A_k \subseteq A_{k+n}$,
there is no $(a_0,...,a_{N-1}) \in
   \prod_{i=0}^{N-1} \hat{A}_{k+i}$ satisfying the first
alternative. So  setting $w_N(x)=w(x)$ we easily see that
    $(w_n(x))_{n=0}^{N} \leq _k \vs$
   and $$<(w_n(x))_{n=0}^{N} \; \| \; (A_{k+n})_{n=0}^{N}>_c \subseteq
   E^c.$$ The inductive step of the construction of $\vw$  is
   complete and as we have already mentioned in the beginning of
   the proof this leads to a contradiction.
\end{proof}

The next lemma is crucial for the proof of Theorem \ref{thm3} and
it is the second and last point of the proof where the
Hales--Jewett theorem is utilized.

\begin{lem}\label{Cl32}
   Let $k \in \nn$, $E \subseteq W(A)$ and $\vs=(s_n(x))_{n=0}^ \infty \in \V$
    such that $E$ is $k$-large in $\vs$. Then there exist
    $ m\geq1$,
    $w(x) \in <(s_n(x))_{n=0}^{m-1}\ \| \   (A_{k+n})_{n=0}^{m-1}>_v$
   and $\vt \in \V$ with $\vt \leq _{k+m} (s_{n}(x))_{n=m}^\infty$ such that
    setting $F=\{w(a):a \in A_{k}\}$ then
   $E\cap E_F$ is $(k+m)$-large in $\vt$.
\end{lem}

\begin{proof}
   Fix  $k \in \nn$, $E \subseteq W(A)$ and $\vs=(s_n(x))_{n=0}^ \infty \in \V$
and assume  that $E$ is $k$-large in $\vs$.

\begin{claim}\label{Claim33}
    There exist $r \in \nn$ and a finite sequence
   $(w_i(x))_{i=0}^r$ with
    $(w_i(x))_{i=0}^r \leq _k \vs$ such that for every  $w(x)\in V(A)$
    with  $(w_0(x),..., w_r(x),w(x)) \leq _k
    \vs$
   there exist $v \in <(w_n(x))_{n=0}^r \  \| \ (A_{k+n})_{n=0}^r>_c$ and
   $a \in A_{k+r+1}$
   such that
   $v\; w(a) \in E.$
\end{claim}

 \begin{proof}[Proof of Claim \ref{Claim33}]  Assume that the conclusion fails.
 By induction we easily  construct  a sequence of variable words
  $\vec{w} =(w_n(x))_{n=0}^ \infty \in \V$ with $\vw \leq_k\vs$ and
  $w_0(x)=s_0(x)$ such
  that for every $n \in \nn$, every $b \in <(w_i(x)_{i=0}^n \; \| \;
  (A_{k+i})_{i=0}^n>_c
 $ and every $a \in A_{k+n+1}$ we have that $bw_{n+1}(a) \in
 E^c.$
 Setting for every $n \in \nn$, $w^{\; \prime} _n(x) = w_{2n}(x)
 w_{2n+1}(x)$ and $\vw ^\prime = (w^{\prime} _n (x))_{n=0}^ \infty$
 then we obtain that
  $\vw ^\prime \leq _k \vs $ and
 $<\vw ^\prime \;\|\; (A_{k+n})_{n=0}^ \infty>_c \subseteq E^c$, a
 contradiction.
\end{proof}

\begin{claim}\label{Claim34} Let  $(w_i(x))_{i=0}^r$ be as  in
Claim \ref{Claim33}. Let
$n_0 \geq 1$ be the least integer such that $w_0(x) ... w_r(x)
\in <(s_n(x))_{n=0}^{n_0-1}\;\|\; (A_{k+n})_{n=0}^{n_0-1}>_v$ and
set
$$  q=\big | <(w_n(x))_{n=0}^r \  \| \ (A_{k+n})_{n=0}^r>_c
\big |\cdot |A_{k+r+1}|,$$ $$\ N=HJ(|A_k|,q) \ \text{and}
   \ m=n_0+N,$$
   where $HJ(|A_k|,q)$ is as in Theorem \ref{thm1}.

Then for every variable word
 $v(x) \in  <(s_n(x))_{n=m}^\infty \;\|\;(A_{k+n})_{n=m}^\infty>_v$ there
 exist a constant word
  $$w\in<(s_n(x))_{n=0}^{n_0-1}\;\|\;  (A_{k+n})_{n=0}^{n_0-1}>_c,$$
  a letter $a \in A_{k+r+1}$  and  a variable word
$$y(x) \in [(s_n(x))_{n=n_0}^{n_0+N-1}\;\| \; A_{k}]_v$$
 such that
    $ w y(b)  v(a) \in E$  for all  $b \in A_k$.
\end{claim}

 \begin{proof}[Proof of Claim \ref{Claim34}] We set $B= A_k$ and we
 fix $v(x) \in <(s_n(x))_{n=m}^\infty
 \;\|\;(A_{k+n})_{n=m}^\infty>_v$. We will  define a  finite coloring of
 $B^N$, depending on  $v(x)$,
 as follows.
Let $\mathbf{b}=(b_0,...,b_{N-1}) \in
 B^N$ be arbitrary and let
  $z(x)= s_{n_0}(b_0) ...
   s_{n_0+N-1}(b_{N-1}) v(x).$
Since $B=A_k\subseteq A_{k+n}$ for all $n\in\nn$, we have that
$$z(x) \in <(s_n(x))_{n=n_0}^\infty
\;\|\;(A_{k+n})_{n=n_0}^\infty>_v.$$
  Therefore, $(w_0(x),..., w_r(x),z(x))$ is a finite
  extracted $k$-block subsequence of $\vs$. Hence, by Claim \ref{Claim33}
  there exists
    $w^{\mathbf{b}} \in <(w_n(x))_{n=0}^r \  \| \ (A_{k+n})_{n=0}^r>_c$
   and $a^{\mathbf{b}} \in A_{k+r+1}$
    such that
   $w^{\mathbf{b}} z(a^{\mathbf{b}}) \in E.$
  We define
  $$c_{v(x)}:B^N\to\; <(w_n(x))_{n=0}^r \  \| \
(A_{k+n})_{n=0}^r>_c \times A_{k+r+1} $$ by setting
\[c_{v(x)}(\mathbf{b})=(w^{\mathbf{b}},a^{\mathbf{b}}).\]
Since $q=\big | <(w_n(x))_{n=0}^r \  \| \ (A_{k+n})_{n=0}^r>_c \big
|\cdot |A_{k+r+1}|$ and $N=HJ(|B|,q)$, by the Hales-Jewett
Theorem and the way we define $c_{v(x)}$, we conclude that  there
exist a variable word $$ y(x) \in
[(s_n(x))_{n=n_0}^{n_0+N-1}\;\| \; B]_v,$$ a constant word
 $w \in <(w_n(x))_{n=0}^r \  \| \ (A_{k+n})_{n=0}^r>_c$ and
 $a \in A_{k+r+1}$
such that
 $$w y(b)v(a) \in E,$$ for all $b \in B$.
 Since $w_0(x)  ... w_r(x) \in
   <(s_n(x))_{n=0}^{n_0-1}\;\|\; (A_{k+n})_{n=0}^{n_0-1}>_v$ we
   easily see that  $w\in<(s_n(x))_{n=0}^{n_0-1}\;\|\;
   (A_{k+n})_{n=0}^{n_0-1}>_c$ and  $wy(b)v(a)\in
   E$, for all $b\in B$.
   \end{proof}

\begin{claim} \label{Claim35} Let  $(w_i(x))_{i=0}^r$ be as  in
Claim \ref{Claim33} and
$n_0$,  $N$ and $m$ be as in Claim \ref{Claim34}. Let
$$\mathbb{P}=<(w_n(x))_{n=0}^r\;\|\; (A_{k+n})_{n=0}^r>_c \times
[(s_n(x))_{n=n_0}^{n_0+N-1}\;\| \; A_{k}]_v$$
and for every $(w,y(x))\in\mathbb{P}$ let $F_{(w,y(x))}=\{w y(a) :
a \in A_k\}$. Finally, let
\[E^*=\bigcup_{(w,y(x))\in\mathbb{P}} E \cap E_{F_{(w,y(x))}},\]
where $E_{F_{(w,y(x))}}=\{z\in W(A): uz \in E \text{ for all }
u\in F_{(w,y(x))}\}$.

Then  $E^*$ is $(k+m)$-large in
 $(s_n(x))_{n=m}^\infty$.
 \end{claim}

\begin{proof}[Proof of Claim \ref{Claim35}]
Let $\vt=(t_n(x))_{n=0}^ \infty$ be an arbitrary extracted $(k+m)$-block
subsequence  of
$(s_n(x))_{n=m}^\infty$. Since $E$ is
 $k$-large in $\vs$, by Lemma \ref{exlemma} we have that $E$
is $(k+m)$-large in $(s_n(x))_{n=m}^\infty$ and hence  $E$ is
$(k+m)$-large in $\vt$.  By Lemma \ref{Cl31}
there exists a variable word
 $v(x)$ in $<(t_n(x))_{n=0}^ \infty \  \| \ (A_{k+m+n})_{n=0}^ \infty>_v$ such that
for every $a \in A_{k+m}$, $v(a) \in E$. Since $v(x) \in
<(t_n(x))_{n=0}^ \infty \ \| \ (A_{k+m+n})_{n=0}^ \infty>_v$ we
easily obtain that $v(x)
\in <(s_n(x))_{n=m}^\infty \;\|\;(A_{k+n})_{n=m}^\infty>_v$ and so
by Claim \ref{Claim34} there exists a pair $(w,y(x))\in\mathbb{P}$
and $a \in A_{k+r+1}\subseteq A_{k+m}$ such that $wy(b)v(a)\in
E$, for all $b\in A_k$. Hence $v(a) \in E\cap
E_{F_{(w,y(x))}}\subseteq E^*$. Moreover, $v(a) \in
<(s_n(x))_{n=m}^\infty \;\|\;(A_{k+n})_{n=m}^\infty>_c$.
Therefore, $E^*$ is $(k+m)$-large in
 $(s_n(x))_{n=m}^\infty$.
\end{proof}

We are now ready to finish the proof of the lemma. Indeed, by
Claim \ref{Claim35} and  Lemma \ref{Cl30}, there exist a pair
$(w_0,y_0(x))\in\mathbb{P}$ and an extracted $(k+m)$-block subsequence $\vt$
of $(s_n(x))_{n=m}^\infty$ such that $E\cap E_{F_{(w_0,y_0(x))}}$
is $(k+m)$-large in $\vt$. We set $$w(x) =w_0 y_0(x) \text{ and }
F=F_{(w_0,y_0(x))}=\{w(a): a\in A_k\}.$$ Then $w(x) \in
<(s_n(x))_{n=0}^{m-1}\;\|\; (A_{k+n})_{n=0}^{m-1}>_v$ and $E\cap
E_F$ is $(k+m)$-large in $\vt$, as desired.
\end{proof}

Using Lemma \ref{Cl27}  and arguing as in  Lemma \ref{Cl20} we
have the following.

\begin{lem}\label{l18}  Let $k \in \nn$, $E \subseteq W(A)$ and
 $\vs= (s_n(x))_{n=0}^ \infty \in \V$ such that
   $E$ is $k$-large in $\vs$. Then there exist a sequence
   $(w_n(x))_{n=0}^ \infty$ of variable words and two
strictly increasing sequences $(k_n)_{n=0}^ \infty$ and
$(p_n)_{n=0}^ \infty$ in $ \nn $,
with $k_0=k$ and $p_0=0$ such that setting for every $n \in \nn$,
$F_n=< (w_i(x))_{i=0}^n\;\|\;
 (A_{k_i})_{i=0}^{n}>_c$ then
for every $n \in \nn$ the
following properties are satisfied.
\begin{enumerate}
\item[(S1)] $k+p_n \geq k_n$. \item[(S2)]  $w_n(x) \in
<(s_i(x))_{i=p_n}^{p_{n+1}-1} \; \| \;
(A_{k+i})_{i=p_n}^{p_{n+1}-1}>_v$. \item[(S3)]  $E\cap E_{F_n}$ is
$k_{n+1}$-large in $(w_{i}(x))_{i=n+1}^\infty$.
\end{enumerate}
\end{lem}

\begin{lem}\label{l188}
 Let $\vec{w}=(w_n(x))_{n=0}^ \infty$, $(k_n)_{n=0}^ \infty$
 and  $(p_n)_{n=0}^ \infty$ be the sequences obtained in Lemma
 \ref{l18}. Also let $F_n=< (w_i(x))_{i=0}^n\;\|\;
 (A_{k_i})_{i=0}^{n}>_c,$ for all $n \in \nn$.
Then there exist a strictly increasing sequence
 $(r_n)_{n=0}^ \infty$ in $\nn$ with $r_0 =0$
 and a sequence $\vt=(t_n(x))_{n=0}^ \infty$  of variable words
 such that
for every  $n\geq 1$ the following are
  satisfied.
 \begin{enumerate}

 \item[(T1)] $t_{n}(x) \in
 <(w_{r_n+i}(x))_{i=0}^{r_{n+1}-r_{n}-1}
  \; \|\; (A_{k_{r_{n}}+i})_{i=0}^{r_{n+1}-r_{n}-1}>_v$.

 \item[(T2)] For every $a \in A_{k_{r_n}}$ we have
 $t_n(a) \in E$ and $u t_{n}(a) \in
 E$,
 for every $u\in$ $ <(t_i(x))_{i=0}^{n-1}   \; \|\;
 (A_{k_{r_{i}}})_{i=0}^{n-1}>_c$.

  \item[(T3)] $<(t_i(x))_{i=0}^n \; \| \;
  (A_{k_{r_{i}}})_{i=0}^n>_c \subseteq
  F_{r_{n+1}-1}$.

 \end{enumerate}
\end{lem}

 \begin{proof} We proceed by induction to the construction
 of the desired sequences.
 For $n=1$ we set $r_0=0$, $r_1=1$ and  $t_0(x)=w_0(x)$. Let
   $G_0=<t_0(x)\; \|\; A_{k_{r_0}}>_c.$ Then
   $G_0 =F_0$ and therefore by Lemma \ref{l18} we have that
   $E \cap E_{G_0}$ is
$k_{1}$-large in $(w_{i}(x))_{i=1}^\infty$.
 By Lemma \ref{Cl31}
  there exist $r \in \nn$ and  a variable word
  $$t_1(x) \in <( w_{r_1+i}(x))_{i=0}^{r }\; \|\;
  (A_{k_{r_1}+i})_{i=0}^{r} >_v,$$
  such that for every $a \in A_{k_{r_1}}$ and
  every $u \in <t_0(x)\; \|\; A_{k_{r_0}}>_c$, we have
  $t_1(a)\in E$ and $u t_1(a) \in E$.
We set $r_2=r_1+r+1$.  It is straightforward  that
 $$t_1(x) \in <( w_{r_1+i}(x))_{i=0}^{r_2-r_1-1 }\; \|\;
 (A_{k_{r_1}+i})_{i=0}^{r_2-r_1-1 } >_c,$$
and $$ <(t_i(x))_{i=0}^1\; \|\;
   (A_{k_{r_{i}}})_{i=0}^1>_c\subseteq < (w_i(x))_{i=0}^{r_2-1}\;\|\;
   (A_{k_i})_{i=0}^{r_2-1}>_c= F_{r_2-1}.$$
By the above we have that  (T1) - (T3) are satisfied for $n=1$.
Assume that the construction has been carried out up to some $n
\geq 1$.
 We set $$G_n =<(t_i(x))_{i=0}^n \; \| \; (A_{k_{r_{i}}})_{i=0}^n>_c$$
 and
 by our inductive hypothesis we have that $G_n \subseteq
  F_{r_{n+1}-1}$. Therefore by Lemma \ref{l18} we have that
  $E\cap E_{G_n}$ is $k_{r_{n+1}}$-large in $(w_{i}(x))_{i=r_{n+1}}^\infty$.
By Lemma \ref{Cl31}
  there exist $r \in \nn$ and  a variable word
  $$t_{n+1}(x) \in <( w_{r_{n+1}+i}(x))_{i=0}^{r }\; \|\;
  (A_{k_{r_{n+1}}+i})_{i=0}^{r} >_v,$$
  such that  for every  $a \in A_{k_{r_{n+1}}}$ and every
  $u \in <(t_i(x))_{i=0}^n \; \| \; (A_{k_{r_{i}}})_{i=0}^n>_c$,
 we have  $t_{n+1}(a)\in E$ and  $u t_{n+1}(a) \in E$. We set
 $r_{n+2}=r_{n+1}+r+1$ and therefore
 $$t_{n+1}(x) \in <( w_{r_{n+1}+i}(x))_{i=0}^{r_{n+2}-r_{n+1}-1 }\; \|\;
 (A_{k_{r_{n+1}}+i})_{i=0}^{r_{n+2}-r_{n+1}-1 } >_c,$$
and $$ <(t_i(x))_{i=0}^{n+1}\; \|\;
   (A_{k_{r_{i}}})_{i=0}^{n+1}>_c\subseteq
   < (w_i(x))_{i=0}^{r_{n+2}-1}\;\|\; (A_{k_i})
   _{i=0}^{r_{n+2}-1}>_c= F_{r_{n+2}-1}.$$
It is easy to check that (T1) - (T3) are satisfied and the proof
is complete.
\end{proof}

The following Corollary as well as its proof is similar to Corollary \ref{corol22}.
For completeness we include its proof.

\begin{cor}\label{l19}
   Let $k \in \nn$, $E \subseteq W(A)$ and $\vs \in \V$ such that
   $E$ is $k$-large in $\vs$.
   Then there exists an extracted
   $k$-block subsequence $\vec{u}$ of $\vs$
   such that
    $$<\vec{u} \;\|\; ( A_{k+n})_{n=0}^ \infty>_c \subseteq E.$$
\end{cor}

\begin{proof}
Let $(r_n)_{n=0}^ \infty$ and $\vt= (t_n(x))_{n=0}^ \infty$ be
the sequences obtained in
Lemma \ref{l188}. Using (S1) and (S2) of Lemma \ref{l18} and
(T1) of Lemma \ref{l188} we have that
    $$t_n(x) \in  <(s_i(x))_{i=p_{r_n}}^{p_{r_{n+1}}-1} \;
   \| \; (A_{k+i})_{i=p_{r_n}}^{p_{r_{n+1}}-1}>_v,$$
 i.e. $\vt$ is an extracted $k$-block subsequence of $\vs$.

For every $ n \in \nn$,
We set $u_n(x) =t_{2n} (x) t_{2n+1}(x) $ and let $\vec{u} = (u_n(x))_{n=0}^ \infty$.
Clearly,  $\vec{u}$ is an extracted $k$-block subsequence of $\vs$.
Since the
sequences $(k_n)_{n=0}^ \infty$ and $(r_n)_{n=0}^ \infty$  are
increasing we obtain that
$$ <(t_i(x))_{i=0}^n \; \| \;
(A_{k+i})_{i=0}^n>_c \subseteq <(t_i(x))_{i=0}^{n}   \; \|\;
 (A_{k_{r_{i}}})_{i=0}^{n}>_c,$$ for all $n \in \nn$.
 By the definition $\vec{u}= (u_n(x))_{n=0}^ \infty$
we obtain that
$$ <(u_i(x))_{i=0}^n \; \| \;
(A_{k+i})_{i=0}^n>_c \subseteq <(t_i(x))_{i=0}^{2n+1}   \; \|\;
 (A_{k+i})_{i=0}^{2n+1}>_c.$$
 Hence,
 $$ <(u_i(x))_{i=0}^n \; \| \;
(A_{k+i})_{i=0}^n>_c \subseteq <(t_i(x))_{i=0}^{2n+1}   \; \|\;
 (A_{k_{r_{i}}})_{i=0}^{2n+1}>_c,$$ for all $n \in \nn$.
 By (T2) of Lemma \ref{l188} we have that for all $n \in \nn$,
$$ <(u_i(x)_{i=0}^n \;\|\; ( A_{k+i})_{i=0}^n>_c \subseteq E,$$
and therefore  $<\vec{u} \;\|\; ( A_{k+n})_{n=0}^ \infty>_c \subseteq E.$
\end{proof}

\begin{thm}\label{c20}
   Let $k \in \nn$ and  $\vec{v}=(v_n(x))_{n=0}^ \infty$ be a
   sequence of variable words.
   Let $r\geq 2$ and  $<\vec{v}\;\|\; (A_{k+n})_{n=0}^ \infty>_c = \cup_{i=1}^r
  E_i$. Then
   there exist $1 \leq i \leq r$ and
   an extracted $k$-block subsequence $\vec{u}=(u_n(x))_{n=0}^ \infty$
   of $\vec{v}$
   such that
   $<\vec{u}\;\|\; (A_{k+n})_{n=0}^ \infty>_c \subseteq E_i$.
\end{thm}

 \begin{proof}
Let $\vw \in \V$ such that
      $\vw \leq_k \vec{v}$.   By Fact \ref{Cf25}, we have that
     $$<\vw \  \| \ (A_{k+n})_{n=0}^ \infty>_c
    \subseteq <\vec{v} \  \| \ (A_{k+n})_{n=0}^ \infty>_c$$ and
    therefore $<\vw \  \| \ (A_{k+n})_{n=0}^ \infty>_c
    \subseteq \cup_{i=1}^r
  E_i$. Thus,  trivially, we get that $\cup_{i=1}^r E_i$ is
  $k$-large in
  $\vec{v}$.
Hence, by Lemma \ref{Cl30}, there
  exist $1 \leq i \leq r$ and $\vs \leq _k \vec{v} $
  such that $E_i$ is $k$-large in $\vs$ and so
  by Corollary \ref{l19},
  there exists a
   $\vec{u} \leq _k \vs$
   such that
    $$<\vec{u}\;\|\; ( A_{k+n})_{n}>_c \subseteq E_i.$$
    Since  $\vec{u} \leq _k \vs\leq _k \vec{v}$, by Fact \ref{Cf26} we
    have that $\vec{u} \leq _k \vec{v}$ and the proof is complete.
\end{proof}

Theorem \ref{c20} easily yields Theorem \ref{thm3}.

\begin{proof}[Proof of Theorem \ref{thm3}] Let $r \geq 2 $ and
let $W(A) =
\cup _{i=1}^r E_i$. Let $\vec{v}=(x,x,...)$. Then
 $<\vec{v}\;\|\; (A_{n})_{n=0}^ \infty>_c = W(A)=\cup_{i=1}^r E_i$.
 Applying  Theorem
   \ref{c20} for $k=0$, we obtain $1 \leq i \leq r$ and
   $\vec{w}=(w_n(x))_{n=0}^ \infty\in \V$
    such that
   $<\vec{w}\;\|\; (A_{n})_{n=0}^ \infty>_c \subseteq E_i$.
\end{proof}

\end{document}